\documentclass[11pt]{amsart}

\usepackage{graphicx}
\usepackage{amsmath}
\usepackage{amssymb}
\usepackage{mathrsfs} 
\usepackage{xcolor}

\usepackage[utf8]{inputenc}

\newtheorem{thm}{Theorem}[section]

\theoremstyle{definition}
\newtheorem{definition}[thm]{Definition}

\theoremstyle{remark}

\newcommand{\R}{\mathbb{R}}

\newcommand{\QU}{\mathbb{H}}
\newcommand{\QS}{\mathbb{S}^{4n-1}}  

\newcommand{\ii}{\textbf{i}}
\newcommand{\jj}{\textbf{j}}
\newcommand{\kk}{\textbf{k}}
\newcommand{\N}{\mathbb{N}}

\newcommand{\LB}{\Delta_{\mathbb{S}^{4n-1}}}

\newcommand{\Hhm}{\mathcal{H}_{h,m}}
\newcommand{\LLB}{\lambda_{h,m}^{\Delta_{\mathbb{S}^{4n-1}}}}
\newcommand{\LG}{\lambda_{h,m}^{\Gamma}}

\newcommand{\dH}{\dim_{\mathcal{H}}}

\DeclareMathOperator{\W}{WF}

\DeclareMathOperator{\spann}{span}

\DeclareMathOperator{\Char}{Char}

\begin{document}

\title[Dimensional estimates for measures on quaternionic spheres]{Dimensional estimates for measures on quaternionic spheres}

\subjclass[2020]{}

\author{Rami Ayoush}
\address{Institute of Mathematics, Universitreiy of Warsaw, Banacha 2, 02-097, Warsaw, Poland}

\email{r.ayoush@uw.edu.pl}

\author{Michał Wojciechowski}
\address{Institute of Mathematics, Polish Academy of Sciences, \\ Śniadeckich 8, \\ 00-656 Warsaw, Poland}
\email{miwoj@impan.pl}
\thanks{R.A. and M.W. were supported by the National Science Centre, Poland, CEUS programme, project no. 2020/02/Y/ST1/00072.}

\begin{abstract}
In this article we provide lower bounds for the lower Hausdorff dimension of finite measures assuming certain restrictions on their quaternionic spherical harmonics expansion. This estimate is an analog of a result previously obtained by the authors for the complex spheres.
\end{abstract}

\maketitle


\section{Introduction} Let $n\geq 2$ be a fixed integer. We denote by $\QS \subset \QU^{n} \simeq \R^{4n}$ the quaternionic unit sphere, i.e.
\[
\QS = \{x \in \QU^{n}: \lVert x \rVert = 1\},
\]
where the norm comes from the quaternionic inner product:
\[
\langle x,y \rangle \quad = \sum_{i=1}^{n} x_{i}\bar{y_{i}} \quad \text{ for } x=(x_{1},\dots,x_{n}), \ y=(y_{1},\dots,y_{n}) \in \QU^{n}.
\]
Here, the conjugation of a quaternion
\[
x = a+b\ii+c\jj+d\kk, \quad a,b,c,d\in \R
\]
is given by
\[
x = a-b\ii-c\jj-d\kk.
\]
We also denote $\Re x = a.$

Our results will be stated in terms of the spectral decomposition of $L^{2}(\QS)$ with respect to the Laplace-Beltrami operator on $\QS$, $\LB$, and a sublaplacian
\[
\Gamma = -(T^{2}_{i}+T^{2}_{j}+T^{2}_{k}),
\]
where 
\[
T_{i} f(x) = \frac{d}{dt}\bigg{|}_{t=0} f(exp(-\ii t)x),
\]
and $T_{j}, T_{k}$ are defined analogously, with a use of remaining imaginary parts. 
This decomposition was obtained in \cite{ACMM} in the following theorem (see Proposition 2.1. and Proposition 3.1. therein).
\begin{definition}
	Let us denote 
	\[
	I_{\QU} = \{(h,m) \in \N^{2}: 2m \leq h\}.
	\]
\end{definition}
\begin{thm}[\cite{ACMM}]
	There exist finite dimensional, pairwise orthogonal spaces $\Hhm \subset L^{2}(\QS)$, $(h,m) \in I_{\QU}$, such that
	\[
	L^{2}(\QS) = \widehat{\bigoplus_{(h,m)\in I_{\QU}}} \Hhm
	\]
	and $\Hhm$ are eigenspaces of $\LB$ and $\Gamma$ with eigenvalues 
	\[
	\LLB = h(h+4n-2), \quad \LG = (h-2m)(h-2m+2),
	\]
	respectively.
	
	Moreover, the integral kernel of the orthogonal projection $\pi_{h,m}: L^{2}(\QS) \to \Hhm$ is given by
	\[
	K_{h,m}(x,y) = \frac{(h-2m+1)(h+2m-1)}{(2n-2)(2n-1)}\ {h-m+2n-2\choose 2n-3} \ |\langle x,y \rangle|^{h-2m}
	\]
	\[
	\times J_{m}^{(2n-3,h-2m+1)}(|\langle x,y \rangle|^{2}) U_{h-2m}\Big{(}\frac{\Re \langle x,y \rangle}{|\langle x,y \rangle|}\Big{)},
	\]
	where $J_{m}^{(2n-3,h-2m+1)}$ is the Jacobi polynomial and $U_{h-2m}$ is a suitable Chebyshev polynomial of the second kind. 
\end{thm}
Since $K_{h,n}$ is a continuous function, $\pi_{h,m}$ can be extended to a projection from the space of finite measures.

Our main result is the following:
\begin{definition}
	For $\epsilon > 0$ let us denote by $C(\epsilon)$ the set
	\[
	C(\epsilon) = \bigg{\{}(h,m) \subset I_{\QU}: \Big{|}\frac{m}{h} - \frac{1}{2}\Big{|} < \epsilon \bigg{\}}.
	\]
\end{definition}
In the definition above we assume that zero is a natural number.
\begin{definition}
For a finite measure $\mu \in M(\QS)$ we denote by
\[
\dH(\mu) = \inf \{\dH(F): \mu(F) \neq 0\}
\]
the lower Hausdorff dimension of $\mu$.
\end{definition}
\begin{thm}\label{mainres}
	Let $\mu \in M(\QS)$ be a finite measure. Suppose that for some $\epsilon > 0$ the set
	\begin{equation}\label{assumption}
		C(\epsilon) \cap \{(h,m): \pi_{h,m} \mu \not\equiv 0\}
	\end{equation}
	is finite. Then
	\[
	\dH(\mu) \geq 4n - 4.
	\]
\end{thm}
The theorem above is proved by using microlocal techniques in combination with properties of some special sets from Harmonic Analysis of measures, namely Riesz sets (see \cite{Mey}) and $s$-Riesz sets (see \cite{RW}). Such an approach to the regularity of measures was first applied in \cite{Bru} in the study of absolute continuity and then developed for the purpose of Hausdorff dimension estimates in \cite{AWmicro}. Results from \cite{AWmicro} generalize classical results concerning regularity of pluriharmonic measures due to Aleksandrov and Forelli (see Theorem 3.1.2. in \cite{Ale} and Corollary 1.11. in \cite{For}).
\section{Microlocal toolbox} Let $\nu \in \mathcal{D}'(\R^{n})$ be a distribution. We define the wave front set of $\nu$ as
\[
\W(\nu) = \{(x,\xi) \in \R^{n}\times \R^{n}\setminus\{0\}: \xi \in \W_{x}(\nu)\}
\]
and 
\[
\W_{x}(\nu) = \bigcap_{\phi \in C^{\infty}_{0}(\R^{n}), \phi(x) \neq 0} \Sigma(\phi \nu),
\]
where for any compactly supported distribution $\mu$, the set $\Sigma(\mu)$, is the complement (in $\R^{n}\setminus\{0\}$) of
\begin{multline*}
	\{\xi \in \R^{n}\setminus\{0\}: \text{there exists a conic neighbourhood of } \xi,~ C_{\xi}, \text{ such that } \\ \forall N>0 ~ \exists ~C_{N} \text{ such that} ~ |\widehat{\mu}(\xi')| \leq C_{N}(1+|\xi'|)^{-N} \quad \forall ~ \xi' \in C_{\xi}\}.
\end{multline*} 
Wave front set can be also defined on cotangent bundles of abstract manifolds by using coordinate charts (see Theorem 8.2.4. in \cite{Hor1} and comments following it). In particular, since the coordinate charts can be chosen to be bilipschitz (and thus preserving the Hausdorff dimension), in our case of quaternionic sphere most of the properties of wave fronts can be just verified on $\R^{4n-1}$.

It turns out that in the case when the distribution is a Radon measure, the size and the shape of a wave front set give some information about its lower Hausdorff dimension.
\begin{definition}[\cite{AWmicro}]
Let $V\subset \R^{n}$ be a $k$-dimensional subspace. We say that a set $F \subset \R^{n}$ has a $k$-dimensional gap given by $V$ when for each $a\in \R^{n}$
\[
(V+a)\cap F \text{ is a bounded set.}
\]
\end{definition}
\begin{thm}[\cite{AWmicro}, Theorem 1.5]\label{AW} Let $\mu \in M(\R^{n})$ be a finite measure. If the condition
\[
\W_{x}(\mu) \text{ has a $k$-dimensional gap}  
\]
is fulfilled at $\mu$-almost every $x$, then
\[
\dH(\mu) \geq k.
\]
\end{thm}
We will need two more results from microlocal analysis. The first is 

\begin{thm}\label{mic}
Suppose that $A$ is an $s$-th order, properly supported pseudodifferential operator on a compact manifold $M$ and  $\nu \in \mathcal{D}'(M)$. Then
\begin{equation} \label{elliptic}
	\W(\nu) = \W(A\nu) \cup \Char(A),
\end{equation}
where 
\[
\Char(A) = \{(x,\xi) \in T^{*}M\setminus\{0\}: \sigma(A)(x,\xi)=0\}
\]
and $\sigma(A)$ stands for the principal symbol of $A$.
\end{thm} 

For the details and proof see \cite{Hor3}, Theorem 18.1.28.

The second is a theorem of Strichartz concerning functional calculus of pseudodifferential operators.

\begin{thm}[\cite{Stri}]\label{functional} If $L_{1}, \dots, L_{m}$ are first order pseudodifferential operators on a smooth manifold $M$ such that
	\begin{itemize}
	\item 
	\[ \int_{M} L_{j}u(x) \overline{v(x)}dx =  \int_{M} u(x) \overline{L_{j}v(x)}dx\]
	for every $j$ and all $u,v \in C^{\infty}(M)$, \\
	\item \[ L_{i}L_{j} = L_{j}L_{i} \quad \text{for all }i,j\]\\
	\item $Q := L^{2}_{1}+\dots+L^{2}_{m}$ is eliptic in the sense that in each coordinates its symbol satisfies
	\[
	q(x,\xi) \geq c|\xi|
	\]
	for some constant $c>0$.
\end{itemize}
	Moreover, suppose that for $m\in C^{\infty}(M)$, for some fixed $b\geq 0$ we have
	\[
	|(\frac{\partial}{\partial x})^{\alpha}m(x)| \leq C_{\alpha}(1+\sum_{j=1}^{m}x_{j}^{2})^{\frac{1}{2}(b-\sum \alpha_{i})}
	\]
	for all multiindices $\alpha$ and some constants $C_{\alpha}$.
	Then, $m(L_{1},\dots,L_{m})$ is a pseudodifferential operator of order $b$ with a principal symbol \\ $m(\sigma(L_{1}),\dots,\sigma(L_{m}))$.
\end{thm}
\section{Proof of the main result}
\begin{proof}[Proof of the Theorem \ref{mainres}]
We proceed similarly as in the proof of Theorem 1.9. in \cite{AWmicro}.

We identify tangent and cotangent bundles over the sphere by using the Riemannian metric inherited from $\R^{4n}$. At each $x \in \QS$ we consider orthogonal decomposition of $T_{x}\QS = V_{1}\oplus V_{2}$, where $V_{2} = \spann_{\R}\{ix, jx, kx\}$ and we introduce coordinates
\[
(\xi_{1}, \xi_{2}) \in T_{x}\QS \simeq V_{1}\times V_{2} \simeq \R^{4n-4}\times\R^{3}
\]
compatible with this decomposition.

We will prove that, under assumption \eqref{assumption} (from Theorem \ref{mainres}), at each $x\in \QS$, $\W_{x} \mu$ has a $(4n-4)$-dimensional gap given by $V_{1}$. To obtain this, let us consider two auxiliary operators
\[
L_{1} = \sqrt{\LB + (2n-1)^{2}Id}-(2n-1)Id,
\]
\[
L_{2} = \sqrt{Id+\Gamma}-Id.
\]
For $(x,\xi)\in T \QS$ we have
\[
\sigma(L_{1})(x,\xi) = |\xi|,
\]
\[
\sigma(L_{2})(x,\xi) = |\xi_{2}|.
\]
They clearly satisfy the assumptions of Theorem \ref{functional} with $b=0$. Moreover, on the space $\Hhm$ the eigenvalues of $L_{1}$ and $L_{2}$ are $h$ and $h-2m$, respectively.
Let $\psi \in C^{\infty}(\R^{2})$ be a function satisfying two conditions:
\begin{itemize}
\item $\psi$ is $0$-homogeneous outside some small neighbourhood of the origin, \\
\item $\psi$ is supported on $C(\epsilon)$ and $\psi \equiv 1$ on $C(\epsilon/2)$.
\end{itemize}
By the theorem of Strichartz (Theorem \ref{functional}), the operator $\psi(L_{1}, \frac{L_{1}-L_{2}}{2})$ is a pseudodifferential operator represented by
\[
L_{3} = \sum_{(h,m)\in I_{\QU}} \psi(h,m) \pi_{h,m} \quad \text{ in } \mathcal{D}'(\QS).
\]
In particular, $L_{3}\mu \in C^{\infty}(\QS)$, thus by Theorem \ref{mic}, we have
\begin{equation}\label{ell}
	\W\mu \subset \Char(L_{3}).
\end{equation}
Thus, it suffices to show that $V_{1}$ determines a $(4n-4)$-dimensional gap in $\Char(L_{3})$. Suppose that this is not true for some $x\in \QS$. This implies that there exists a sequence $(\xi^{j}) \subset T_{x}\QS$ such that
\[
\frac{|\xi_{2}^{j}|}{|\xi_{1}^{j}|} \to 0 \quad \text{as } j\to +\infty. 
\]
On the other hand, from \eqref{ell} we get that
\[
\psi\Big{(}|\xi^{j}|, \frac{|\xi^{j}|-|\xi_{2}^{j}|}{2}\Big{)} = 0,
\]
but
\[
\frac{|\xi^{j}|-|\xi_{2}^{j}|}{2|\xi^{j}|} = \frac{\sqrt{|\xi^{j}_{1}|^{2}+|\xi^{j}_{2}|^{2}}-|\xi_{2}^{j}|}{2\sqrt{|\xi^{j}_{1}|^{2}+|\xi^{j}_{2}|^{2}}} \to \frac{1}{2}\quad \text{as } j\to +\infty, 
\]
which contradicts with the choice of $\psi$.
\end{proof} 

\bibliography{biblio} 
\bibliographystyle{alpha}

\end{document}